\documentclass[a4paper,11pt]{article}
\usepackage{graphicx}
\usepackage{amssymb}
\usepackage[french, english]{babel}
\usepackage[utf8]{inputenc}
\usepackage{amsmath,amsfonts,amssymb,amsthm}
\usepackage[T1]{fontenc}
\usepackage{fullpage}
\usepackage{hyperref}
\usepackage{mathrsfs}
\usepackage{relsize}
\usepackage{tikz}
\usepackage{bbm}
\usepackage{appendix}
\usepackage{pifont}
\usepackage[normalem]{ulem}
\usetikzlibrary{patterns}

\definecolor{darkgreen}{rgb}{0.0, 0.5, 0.0}

\newcommand{\eps}{\varepsilon}

\newcommand{\T}{\mathcal{T}}
\newcommand{\bT}{\mathbf{T}_{2n,g_n}}
\newcommand{\B}{B}
\newcommand{\C}{\mathbf{C}^{(n)}_{\vec{\ell}}(k_1,k_2)}
\newcommand{\interv}{\left[\frac{\theta}{4},\frac{\theta}{2}+\frac 1 4\right]}

\newcommand{\PP}{\mathbf{P}^{(n)}_\ell}

\newcommand{\ksep}{$(k_1,k_2)$-separating multicurve}
\newcommand{\iso}{\mathrm{Iso}_{\eps} \left( \mathbf{T}_{2n,g_n} \right)}
\newcommand{\isostar}{\mathrm{Iso}^*_{\eps} \left( \mathbf{T}_{2n,g_n} \right)}

\newcommand{\E}{\mathbb E}

\renewcommand{\P}{\mathbb P}

\newcommand{\diam}{\mathrm{diam}}

\theoremstyle{definition}

\newtheorem{thm}{Theorem}

\newtheorem{defn}{Definition}
\newtheorem{rem}[defn]{Remark}

\newtheorem{prop}[defn]{Proposition}
\newtheorem{corr}[defn]{Corollary}

\newtheorem{lem}[defn]{Lemma}
\newtheorem{conj}[defn]{Conjecture}

\title{\bf{Distances and isoperimetric inequalities in random triangulations of high genus}}
\author{Thomas \bsc{Budzinski}\footnote{CNRS and ENS Lyon, \url{thomas.budzinski@ens-lyon.fr}}, \, Guillaume \bsc{Chapuy}\footnote{CNRS and Université Paris Cité, \url{guillaume.chapuy@irif.fr}.} \, and Baptiste \bsc{Louf}\footnote{CNRS and Université de Bordeaux, \url{baptiste.louf@math.u-bordeaux.fr}}}

\begin{document}

\maketitle

\begin{abstract}
We prove that uniform random triangulations whose genus is proportional to their size $n$ have diameter of order $\log n$ with high probability. We also show that in such triangulations, the distances between most pairs of points differ by at most an additive constant.
Our main tool to prove those results is an isoperimetric inequality of independent interest: any part of the triangulation whose size is large compared to $\log n$ has a perimeter proportional to its volume.

\end{abstract}

\section{Introduction and main results}

\medskip

\paragraph{Random planar and non-planar maps.} In this paper we study random maps on surfaces, in a regime in which both their size $n$ and their genus $g$ go to infinity. Here and later, a \emph{map} is a finite graph embedded on an oriented compact surface, considered up to homeomorphism -- maps can also be thought of as "discrete surfaces" made by gluing finitely many polygons by their sides. Different variants of maps can be considered by fixing the degrees of these polygons. In this paper, we will be interested in \emph{triangulations}, where all faces have degree three.

In the planar case ($g=0$, $n\rightarrow \infty$), random maps are very well understood both locally and globally. Locally, their behaviour is described by random infinite maps such as the Uniform Infinite Planar Triangulation (UIPT,~\cite{AS03}) and its variants~\cite{Kri05, CD06, Bud15}. Globally, it is known that their diameter grows as $n^{1/4}$~\cite{CS04}, and random maps rescaled by $n^{1/4}$ converge to the Brownian map~\cite{LG11, Mie11}, a result which holds for a rich variety of models (e.g.~\cite{Mar16}). Direct approaches to the continuum limit via Liouville quantum gravity have been independently developed~\cite{DKRV16, MS15b}.
The results for $g=0$ partially extend to the case of a fixed genus $g>0$ \cite{Chapuy:profile, BM22}. However, much remains to be done to understand the behaviour of the limiting objects when $g$ increases. The question is already difficult at the enumerative level, and it is deeply linked to the theory of enumeration of surfaces through the topological recursion~\cite{Eynard:book} and to the double scaling limit of matrix models, see e.g.~\cite{Chapuy:voronoi}.
Beyond the planar case, another extreme case which is well understood is the case where the genus is not constrained. This model is of a very different nature
and is almost equivalent to the configuration model studied in random graph theory. The genus tends to concentrate very close to its maximum possible value and random maps only have a logarithmic number of vertices, see~\cite{Gam06, CP16, BCP19}.

\paragraph{The high genus regime.} A much more difficult task is to understand random maps when both $n$ and $g$ go to infinity, in particular in the \emph{high genus regime} where the genus $g$ is proportional to the size $n$. Since the Euler characteristic becomes strongly negative, this setting has long been suspected to result in hyperbolic behaviour. However, it is difficult to approach, in particular because no accurate enumerative estimates are known in this range.
Moreover, this regime is already nontrivial in the case of \emph{unicellular maps} (maps with a single face), for which it was proved that
the local limit is a supercritical random tree~\cite{ACCR13} and the diameter is logarithmic~\cite{Ray13a}.
More recently, another motivation for studying the high genus regime has come from the analogy between high genus random maps and random hyperbolic surfaces of genus $g$ under the Weil-Petersson measure when $g\rightarrow \infty$. These random surfaces are well understood. In particular, Mirzakhani famously proved in~\cite{Mir13} that they have diameter of order $\log g$, and no short separating geodesics in a precise sense. In the  unicellular case, several results or conjectures support this analogy~\cite{Ray13a, JansonLouf1, JansonLouf2, MP19}.
\medskip

Until recently, understanding high genus maps beyond the unicellular case was a wide open problem, see e.g.~\cite[Chapter 6]{B13book}. A first step in this direction was made in the paper~\cite{BL19} by two of the authors, where the \emph{local} behaviour of uniform triangulations in the high genus regime was proved to be described by the Planar Stochastic Hyperbolic Triangulations of~\cite{CurPSHIT}, thus proving a conjecture of Benjamini and Curien. Similar results were later obtained in the case of arbitrary (even) face degrees~\cite{BL20}.  

\paragraph{Global distances in high genus triangulations.} However, so far, the \emph{global} scale for random maps in the high genus regime beyond the unicellular case has been out of reach besides a lower bound result on the \emph{planarity radius}~\cite{Louf} (a discrete analogue of the injectivity radius). 
The main conjecture in this direction, attributable to several authors in the field, has been that the diameter is logarithmic. This is supported by analogy with the unicellular case or with random hyperbolic surfaces, and more recently by the hyperbolic nature of the local limit.

In this paper, we settle this conjecture by the affirmative. As in~\cite{BL19}, we work with triangulations. For the rest of the paper, we fix $0<\theta<\frac{1}{2}$, and let $(g_n)$ be a sequence such that $\frac{g_n}{n}\to \theta$. We let $\bT$ be a triangulation of genus $g_n$ with $2n$ faces chosen uniformly at random.
We write $\text{diam}(\bT)$ for the diameter of its underlying graph, i.e. the maximal graph distance between two of its vertices. We also say that an event holds \emph{with high probability} or \emph{w.h.p.} if it holds with probability tending to $1$ when $n$ tends to infinity.

\begin{thm}[Diameter]\label{thm_diameter}
There exist two constants $c_\theta, C_\theta>0$ depending only on $\theta$ such that
\[c_\theta\log n\leq \text{diam}(\bT)\leq C_\theta\log n\]
with high probability. 
\end{thm}
We refer to Section~\ref{sec:conjectures} for precise conjectures on optimal constants. We also note that it is natural to conjecture an analogous result for the corresponding problem for hyperbolic random surfaces. The natural analogue regime would be to consider Weil-Petersson random surfaces with $n$ cusps and genus $g$, with $g$ and $n$ going to infinity and $g/n$ going to a constant. As far as we know, this regime has not been considered yet and seems difficult. 

The proofs techniques behind Theorem~\ref{thm_diameter} also show the following fact, which roughly says that almost all pairs of points on the same random triangulation are almost at the same distance up to an additive constant. Again, a precise conjecture on the asymptotics of typical distances can be found in Section~\ref{sec:conjectures}.
\begin{thm}[Typical distances]\label{thm_dist}
Conditionally on $\bT$, let $(x_n, y_n, u_n, v_n)$ be four vertices picked uniformly at random in $\bT$, independently of each other. Then the sequence of random variables
\[d(x_n,y_n)-d(u_n,v_n)\]
	is tight, i.e. for all $\eta>0$, there is a constant $M=M_\theta(\eta)$ such that for $n$ large enough, we have
\[ \P \left( \left| d(x_n,y_n)-d(u_n,v_n) \right| \geq M \right) \leq \eta. \]
\end{thm}

\paragraph{Isoperimetric inequalities.} The proofs of the last two theorem rely on studying the growth of balls, in volume \emph{and} perimeter, around vertices of $\bT$. In order to do that, our main tool is an isoperimetric estimate saying that it is not possible to separate the surface into two components of at least logarithmic size by cutting along a small number of edges. We refer to Section~\ref{sec:prel} for a precise definition of a separating multicurve in a triangulation.
	
\begin{thm}[Isoperimetric inequality]\label{thm_isoperimetric}
There are constants $K_{\theta}, \delta_{\theta}>0$ depending only on $\theta$ such that with high probability, for all $K_{\theta}\log n\leq k_1\leq k_2$ and $k_1+k_2=2n$, the map $\bT$ does not contain a multicurve of total length $\ell\leq \delta_{\theta} k_1$ separating it into two connected components with respectively $k_1$ and $k_2$ triangles.
\end{thm}

The idea behind the proof of Theorem~\ref{thm_isoperimetric} is very natural: we will establish a first moment bound on the number of short separating multicurves in $\bT$. For this, we will rely mostly on the coarse enumerative estimates obtained by two of the authors in~\cite{BL19}.
We note that an analogue of Theorem~\ref{thm_isoperimetric} for the high genus Weil-Petersson measure has been proved by Mirzakhani~\cite[Thm.~4.4]{Mir13}, and that the global structure of our proof is quite similar to~\cite{Mir13}. However, the asymptotic estimates that we can rely on are not as precise as in~\cite{Mir13}, and the regime we work with is more intricate as the two parameters $n$ and $g$ grow at the same time.

Our isoperimetric inequality controls the existence of bottlenecks in the graph $\bT$ above the scale $\log n$. We will also prove that this is in some sense "optimal": contrary to the hyperbolic surfaces of~\cite{Mir13}, there exist logarithmic "tentacles" (i.e. subgraphs bounded by only two edges) attached to the graph (see Section~\ref{sec:tentacles}). Overall, this gives us a rough control on the Cheeger constant of $\bT$. We recall that the Cheeger constant of a graph $G$ with vertex set $V$ is defined as
\[ h(G)=\min_{\substack{V_1 \subset V\\ |V_1|\leq|V|/2}} \frac{|\partial_{\mathrm{Ver}} V_1|}{|V_1|},\]
where $|V_1|$ denotes the cardinal of $V_1$ and $|\partial_{\mathrm{Ver}} V_1|$ the number of edges with exactly one endpoint inside $V_1$.

\begin{thm}[Cheeger constant]\label{thm_cheeger}
    There are constants $c'_\theta>c_\theta> 0$ depending only on $\theta$ such that with high probability, the Cheeger constant $h(\bT)$ of $\bT$ satisfies
	\begin{align}\label{eq:thmChegger}
		\frac{c_\theta}{\log n} \leq h(\bT) \leq \frac{c'_\theta}{\log n}.
	\end{align}
\end{thm}
In particular, by the Cheeger inequalities (see e.g.~\cite[Chapter 2]{Ch97}), with high probability the spectral gap of the Laplacian matrix of $\bT$ is $O \left( \frac{1}{\log n} \right)$.
We note that an analogous result for the Weil-Petersson measure was obtained in~\cite{ShenWu}, namely that the spectral gap goes to $0$ if the number $n$ of cusps is $o(g)$ but much larger than $\sqrt{g}$, where $g$ is the genus. It is conjectured in~\cite{ShenWu} that it is still the case if $\frac{n}{g}$ goes to a positive constant, which is also supported by Theorem~\ref{thm_cheeger}.

\paragraph{Structure and main steps of the paper.}
In Section~\ref{sec:prel} we give precise definitions for triangulations and separating multicurves, and we obtain enumerative estimates for ratios of numbers of triangulations in the high genus regime. We also study the concavity of the limiting function governing these estimates, which plays a crucial role in the proof. In Section~\ref{sec:iso} we use all these tools to prove Theorem~\ref{thm_isoperimetric}. In Section~\ref{sec:dist} we study carefully the perimeter and volume growth of balls to deduce Theorems~\ref{thm_diameter} and~\ref{thm_dist} from Theorem~\ref{thm_isoperimetric}. In Section~\ref{sec:tentacles} we study the local "tentacles" of $\bT$ and prove Theorem~\ref{thm_cheeger}. Finally, in section~\ref{sec:conjectures} we state some precise conjectures in the optimal constants describing typical distances and the diameter, and prove that those two constants are not the same if they exist.

\section{Preliminaries}
\label{sec:prel}

\subsection{Definitions}
\label{subsec:def}

A \emph{map} $m$ is a finite graph (with loops and multiple edges allowed) embedded on a compact connected oriented surface, considered up to homeomorphism.
The connected components of the complement of the graph on the surface are called the \emph{faces} of $m$. One may equivalently think of a map as a connected oriented surface made by the side-by-side identification of edges in a finite family of polygons (each polygon becomes a face of the map). The \emph{genus} of a map is the genus of its underlying surface.
The maps that we consider will always be \emph{rooted}, i.e. equipped with a distinguished oriented edge called the \emph{root edge}. The face to the right of the root edge is the \emph{root face}, and the vertex at the start of the root edge is the \emph{root vertex}.

The \emph{degree} of a face in a map is the number of edge-sides incident to it. Note that the two edge-sides of the same edge can be incident to the same face, in which case this edge contributes twice to the degree.
A \emph{triangulation} is a rooted map where all the faces have degree $3$. For every $n \geq 1$ and $g \geq 0$, we will denote by $\T(2n,g)$ the set of triangulations of genus $g$ with $2n$ faces (the number of faces must be even so that the edges can be glued two by two). By the Euler formula, a triangulation in $\T(2n,g)$ has $3n$ edges and $n+2-2g$ vertices. In particular, the set $\T(2n,g)$ is nonempty if and only if $n \geq 2g-1$. 
We will also denote by $\mathbf{T}_{2n,g}$ a uniform random variable on $\T(2n,g)$ and, in accordance with the literature, we will denote by $\tau(n,g)$ the cardinality of $\T(2n,g)$.

A \emph{simple cycle} of length $\ell \geq 1$ in a triangulation $t$ is a finite sequence of oriented edges $(\vec{e}_i)_{0 \leq i \leq \ell}$ of $t$ with $\vec{e}_0=\vec{e}_{\ell}$ such that the starting points of the edges $\vec{e}_i$ are pairwise distinct, and for all $1 \leq i \leq \ell$, the starting point of $\vec{e}_i$ is also the endpoint of $\vec{e}_{i-1}$.

A \emph{multicurve} in a triangulation $t$ is an ordered list of simple cycles $(c_1, \dots, c_s)$ satisfying the following properties:
\begin{itemize}
    \item the cycles $c_i$ are edge-disjoint, i.e. no edge appears in two cycles, even with different orientations;
    \item no two cycles cross each other, i.e. for any vertex $v$ and any four pairwise distinct edges $e_1, e_2, e_3, e_4$ incident to $v$ in this cyclic order, there are no two cycles $c_i, c_j$ such that $c_i$ uses $e_1, e_3$ and $c_j$ uses $e_2, e_4$. See Figure~\ref{fig:multicurve}.
\end{itemize}

Furthermore, we say that the multicurve $(c_1,\dots,c_s)$ is \emph{$(k_1,k_2)$-separating} with $k_1+k_2=2n$ if the complement $t \backslash \left( c_1 \cup \dots \cup c_s \right)$ has exactly two face-connected components with respectively $k_1$ and $k_2$ faces.
Note that in our definitions we require each cycle to be simple, but we allow different cycles in a multicurve to share a vertex. The reason for this is that multicurves of interest to us will be boundaries of balls, in which case the cycles may not be vertex-disjoint, see Figure~\ref{fig:multicurve}.
We will denote by $\vec{\ell}=(\ell_1,\ell_2,\dots,\ell_s)$ the list of lengths of the cycles composing a separating multicurve, and write $|\vec{\ell}|=\ell_1+\ell_2+\dots+\ell_s$.

\begin{figure}
	\centering
\includegraphics[width=0.12\linewidth]{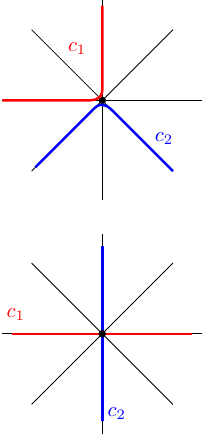}
\hfill
	\includegraphics[width=0.85\linewidth]{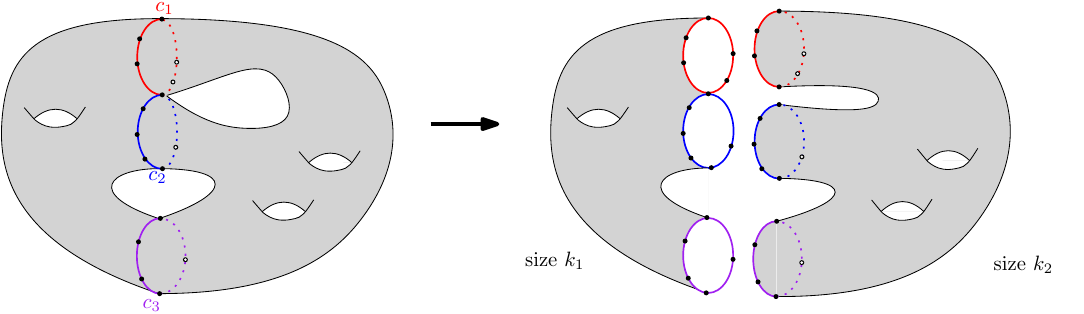}
	\caption{
		Left: relative position of two simple cycles sharing a vertex. The top case is allowed in the definition of a multicurve, the bottom one is not. Right:
		Pictorial view of a $(k_1,k_2)$-separating multicurve $(c_1,c_2,c_3)$ on a triangulation $t$.  The two face-connected components of $t\backslash (c_1\cup c_2\cup c_3)$ have sizes $k_1$ and $k_2$. Note that on this example $c_1$ and $c_2$ share a vertex, which becomes two different vertices in one of the components. The edges and vertices of $t$ which do not appear on the $c_i$ are not represented. }
	\label{fig:multicurve}
\end{figure}

We finally define triangulations of multipolygons, which will roughly be the parts of a triangulation separated by a multicurve. Let $s \geq 0$ and let $\vec{\ell}=(\ell_1,\ell_2,\dots,\ell_s)$ be a length vector (i.e. a finite sequence of positive integers). A \emph{triangulation of the $\vec{\ell}$-gon} is a (non-rooted) map $t$ with $s$ distinguished oriented edges $\vec{e}_1, \dots, \vec{e}_s$ such that:
\begin{itemize}
    \item for all $1 \leq i \leq s$, the face $f_i$ incident to $\vec{e}_i$ on its left has degree $\ell_i$;
    \item the faces $f_i$ are pairwise distinct and have simple, edge-disjoint boundaries (i.e. no vertex appears twice on the boundary of the same face $f_i$, and no edge is incident to two of the faces $f_i$);
    \item all the other faces of $t$ have degree $3$.
\end{itemize}
To prevent any confusion, we insist that in this paper a triangulation of the $\vec{\ell}$-gon is not necessarily planar.
For $s=0$ and $\vec{\ell}=\emptyset$, a triangulation of the $\vec{\ell}$-gon is just a triangulation. The faces $f_i$ are called the \emph{external faces} of $t$, whereas the others are called \emph{internal faces} of $t$. In particular, the union of the boundaries of the external faces form a multicurve. Note that the most usual definition in the literature requires the stronger condition that the boundaries of the $f_i$ are vertex-disjoint. Again, the reason why we relax this condition is that we want to consider a metric ball as a triangulation of a multi-polygon, and the cycles on its boundary may not be vertex-disjoint.

We denote by $\T_{\vec{\ell}}(k,g)$ the set of triangulations of the $\vec{\ell}$-gon of genus $g$ with $k$ internal triangles. If $t\in\T_{\vec{\ell}}(k,g)$, we denote by $\partial t$ the multicurve bounding its external faces, and we write $|t|:=k$ and $|\partial t|:=|\vec{\ell}|$.

We also recall the definition of the \emph{graph distance} in a triangulation $t$. For a pair $(v,v')$ of vertices of $t$, the distance $d_t(v,v')$ is the length of the shortest path of edges of $t$ from $v$ to $v'$. For $r \geq 1$, the ball $B_r(v)$ of radius $r$ and center $v$ is the triangulation of a multipolygon whose internal faces are exactly the faces of $t$ which are incident to at least one vertex $v'$ such that $d_t(v,v') \leq r-1$. We will also denote by $|B_r(v)|$ the number of internal faces of the ball $B_r(v)$ and by $|\partial B_r(v)|$ the sum of the boundary lengths of its external faces. See Figure~\ref{fig:ball}. We conclude with an easy deterministic lemma which will allow us to turn isoperimetric inequalities into distance estimates.

\begin{lem}\label{lem:ball_r+1}
    Let $t$ be a finite triangulation, let $v$ be a vertex of $t$ and let $r \geq 1$ be such that $B_r(v) \ne t$. Then we have
    \begin{equation}\label{eqn:ball_expansion}
    |B_{r+1}(v)| \geq |B_r(v)| + \frac{1}{3} |\partial B_r(v)|.
    \end{equation}
\end{lem}

\begin{proof}
    For each edge $e$ on the boundary of $B_r(v)$, the face $f_r(e)$ of $t$ which is incident to $e$ but is not an internal face of $B_r(v)$ belongs to $B_{r+1}(v) \backslash B_r(v)$. Moreover, each triangular face has $3$ sides, so $e \to f_r(e)$ is at most $3$-to-$1$ and~\eqref{eqn:ball_expansion} follows.
\end{proof}

\begin{figure}
	\centering
	\includegraphics[width=0.45\linewidth]{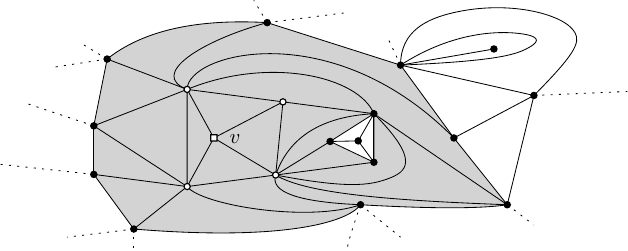}:
	\caption{The ball of radius $r=2$ around a vertex $v$ in a triangulation (in gray). Vertices at distance at most one from $v$ are in white. Note that the boundary of the ball is not connected. We have $|B_r(v)|=20$ and $|\partial B_r(v)|=9+3=12$.}
	\label{fig:ball}
\end{figure}

\subsection{Asymptotic enumeration of high genus triangulations}

Our proofs require good estimates on the numbers of triangulations counted by size and genus. Although an explicit recurrence formula due to Goulden and Jackson~\cite{GJ08} entirely determines these numbers, it does not provide a direct insight into their asymptotic behaviour in the high genus regime. The only asymptotic estimates come from~\cite{BL19}, but they are not written precisely enough for our purposes here. In this section, we recall these estimates and refine them.

We first recall these estimates from~\cite{BL19}. Let $\lambda_c=\frac{1}{12\sqrt{3}}$. For any $\lambda \in (0,\lambda_c]$, let $h \in \left( 0,\frac{1}{4} \right]$ be such that $\lambda=\frac{h}{(1+8h)^{3/2}}$, and let
\begin{equation}
d(\lambda)=\frac{  h \log \frac { 1 + \sqrt { 1 - 4 h } } { 1 - \sqrt { 1 - 4 h } } } { ( 1 + 8 h ) \sqrt { 1 - 4 h } }.
\end{equation}
It was checked in~\cite{BL19} that the function $d(\lambda)$ is increasing with $\lim_{\lambda \to 0} d(\lambda)=0$ and $d(\lambda_c)=\frac{1}{6}$.
For any $\theta \in \left[ 0,\frac{1}{2} \right)$, we denote by $\lambda(\theta)$ the unique solution of the equation
\begin{equation}\label{eqn_lambda_vs_theta}
d(\lambda)=\frac{1-2\theta}{6}.
\end{equation}
In particular, the function $\lambda(\theta)$ is analytic on $\left( 0,\frac{1}{2} \right)$, positive and decreasing on $\left[ 0, \frac{1}{2} \right)$. The next two results were obtained in~\cite{BL19} as a consequence of local convergence results for $\bT$, we restate them here as lemmas.

\begin{lem}[{\cite[Lemma 26]{BL19}}]\label{lem_ratio}
Let $(g_n)_{n \geq 1}$ be a sequence such that $0 \leq g_n \leq \frac{n+1}{2}$ for all $n$ and such that $\frac{g_n}{n}\rightarrow \theta$ for some $\theta \in \left[ 0, \frac{1}{2} \right)$.
Then 
\[\frac{\tau(n-1,g_n)}{\tau(n,g_n)} \xrightarrow[n \to +\infty]{} \lambda(\theta).\]
\end{lem}
Note that this lemma implies that one can write 
$\frac{\tau(n-1,g)}{\tau(n,g)}=\lambda(g/n)+o(1)$ as $n \to +\infty$, where the $o(1)$ is uniform on $g/n \in \left[0, \left( \frac{1}{2}-\eps \right) \right]$, and we will use it under this form\footnote{Indeed, if it was not the case, there would be sequences $(n_k,g_k)$ along which the ratio estimate fails, and we could extract a subsequence along which $g_k/n_k$ converges to $\theta \in \left[ 0,\frac{1}{2}-\eps \right]$, contradicting Lemma~\ref{lem_ratio}.}.

\begin{lem}[{\cite[Thm. 3]{BL19}}]\label{lem_asympto}
Let $(g_n)$ be a sequence such that $0 \leq g_n \leq \frac{n+1}{2}$ for all $n$ and $\frac{g_n}{n} \to \theta \in \left[0, \frac{1}{2} \right]$. Then we have
\[ \tau(n,g_n)=n^{2g_n} \exp \left( f(\theta) n + o(n) \right) \]
as $n \to +\infty$, where $f$ is a continuous function given by $f(0)=\log 12\sqrt{3}$, by $f(1/2)=\log \frac{6}{e}$ and by
\begin{equation}\label{eq_f}
f(\theta)= 2 \theta\log \frac{12\theta}{e} + \theta\int_{2}^{1/\theta} \log \frac{1}{\lambda(1/t)}\mathrm{d}t
\end{equation}
for $0<\theta<\frac{1}{2}$.
\end{lem}

We highlight that the estimate holds on the whole range $\left[ 0,\frac{1}{2} \right]$ of $\theta$. For the same reason as Lemma~\ref{lem_ratio}, this result can also be read as $\tau(n,g)=n^{2g} \exp \left( f \left( \frac{g}{n} \right) n + o(n) \right)$, where the $o(n)$ is uniform in $0 \leq g \leq \frac{n+1}{2}$.

In~\cite{BL19}, Lemma~\ref{lem_asympto} was deduced from Lemma~\ref{lem_ratio}. In particular, when we want to estimate a ratio of the form $\frac{\tau(n',g')}{\tau(n,g)}$, it provides an estimate up to a factor $e^{o(n)}$. The next result, whose proof also relies on Lemma~\ref{lem_ratio}, roughly means that the factor $e^{o(n)}$ can be replaced by $e^{O(g-g')+o(n-n')}$, which will be much better in the case where $(n',g')$ is close to $(n,g)$.

\begin{prop}\label{prop_asympto_ratio}
For all $\theta \in \left( 0,\frac{1}{2} \right)$, there is a constant $a_{\theta} \in (0,1)$ such that the following holds. Let $(g_n)$ be a sequence such that $0 \leq g_n \leq \frac{n+1}{2}$ for every $n$ and $\frac{g_n}{n} \to \theta$. For all integers $m$ and $h$ satisfying $0\leq m\leq \frac{n}{2}$ and $0\leq h\leq \min \left( \frac{g_n}{2}, \frac{m+1}{2} \right)$, we have
\[
\frac{\tau(n,g_n)}{\tau(n-m,g_n-h)}\geq a_{\theta}^{h}\frac{n^{2g_n}}{(n-m)^{2(g_n-h)}}\exp \left( f\left(\frac{g_n}{n}\right) n -f\left(\frac{g_n-h}{n-m}\right)(n-m)+ o(m) \right),
\]
where the $o(m)$ is uniform in $(m,h)$ as $n \to +\infty$ (that is, it is bounded by $m\eps(n)$ with $\eps(n) \to 0$ as $n \to +\infty$).
\end{prop}

\begin{proof}
The Goulden--Jackson formula~\cite{GJ08} reads
\begin{multline*}
(n+1)\tau(n,g)= 4n(3n-2)(3n-4)\tau(n-2,g-1)+4(3n-1)\tau(n-1,g)\\
+4\sum_{\substack{i+j=n-2\\i,j\geq 0}}\sum_{\substack{g_1+g_2=g\\g_1,g_2\geq 0}}  (3i+2)(3j+2)\tau(i,g_1)\tau(j,g_2)+2\mathbbm{1}_{n=g=1}.
\end{multline*}
Hence, very crudely
\[\tau(n,g)\geq n^2\tau(n-2,g-1)\]
for all $n,g$ with $n \geq 2$.
Hence, by Lemma~\ref{lem_ratio} and using the monotonicity of $\lambda$, for $0 \leq i \leq h-1$, we have 
\[\tau(n,g_n-i)\geq (1+o(1)){n^2}{\lambda\left(\frac{g_n}{n}\right)^{2}}\tau(n,g_n-i-1),\]
where the $o(1)$ is uniform in $i$.
Therefore, using $\frac{g_n}{n} \to \theta<\frac{1}{2}$, we get
\begin{equation}\label{eq_asympto_ratio_1}
\frac{\tau(n,g_n)}{\tau(n,g_n-h)}\geq \left( n \lambda \left( \frac{g_n}{n} \right)\right)^{2h}e^{o(h)} \geq \left( n \lambda \left( \frac{\theta}{2}+\frac{1}{4} \right)\right)^{2h}e^{o(h)},
\end{equation}
where $o(h)$ is uniform in $h$ as $n \to +\infty$.

On the other hand, by Lemma~\ref{lem_ratio}, we have
\begin{align*}
\frac{\tau(n,g_n-h)}{\tau(n-m,g_n-h)}&=\frac{e^{o(m)}}{\prod_{i=0}^{m-1}\lambda\left(\frac{g_n-h}{n-i}\right)}\\
&=\exp\left[(g_n-h)\int^{n/(g_n-h)}_{(n-m)/(g_n-h)}\log\frac{1}{\lambda(1/t)}dt+o(m)\right],
\end{align*}
with $o(m)$ uniform in $(m,h)$.
But by the definition~\eqref{eq_f} of the function $f$, we also have
\begin{align*}
nf\left(\frac{g_n-h}{n}\right)-(n-m)f\left(\frac{g_n-h}{n-m}\right)=(g_n-h)\left[2\log\left(\frac{n-m}{n}\right)+\int^{\frac{n}{g_n-h}}_{\frac{n-m}{g_n-h}}\log\frac{1}{\lambda(1/t)}dt\right],
\end{align*}
so we can write
\begin{align}\label{eq_asympto_ratio_2}
\frac{\tau(n,g_n-h)}{\tau(n-m,g_n-h)}&=\left(\frac{n}{n-m}\right)^{2(g_n-h)}\exp\left(nf\left(\frac{g_n-h}{n}\right)-(n-m)f\left(\frac{g_n-h}{n-m}\right)+o(m)\right).
\end{align}
Finally, since $h \leq \frac{g_n}{2}$, for $n$ large enough we have $\left[\frac{g_n-h}{n},\frac{g_n}{n}\right]\subset \interv$ so, since $f\in C^1 \left( \interv  \right)$, we can write
\begin{equation}\label{eq_asympto_ratio_3}
\left| nf\left(\frac{g_n-h}{n}\right)-nf\left(\frac{g_n}{n}\right)\right|\leq h \times \max_{\interv} |f'|.
\end{equation}
The lemma follows by writing
\[\frac{\tau(n,g_n)}{\tau(n-m,g_n-h)}=\frac{\tau(n,g_n)}{\tau(n,g_n-h)}\frac{\tau(n,g_n-h)}{\tau(n-m,g_n-h)}\]
and by combining equations~\eqref{eq_asympto_ratio_1},~\eqref{eq_asympto_ratio_2} and~\eqref{eq_asympto_ratio_3}. In particular, we can take \[a_{\theta}= \lambda \left( \frac{\theta}{2}+\frac{1}{4} \right)^2 \times \exp \left( -{\max_{\interv} |f'|} \right) \in (0,1).\]
\end{proof}

Finally, we recall~\cite[Lemma 1]{BL19}, which will allow us to transfer the above estimates to triangulations of multipolygons. The proof just consists of triangulating external faces.

\begin{lem}\label{lem_filling_boundaries}
Let $k \geq 1$ and $g \geq 0$. For any length vector $\vec{\ell}=\left( \ell_1, \dots, \ell_s \right)$, we have
\[|\mathcal{T}_{\vec{\ell}}(k,g)|\leq (3k+3|\vec{\ell}|)^{s-1}\tau\left(\frac{k+|\vec{\ell}|}2,g\right).\]
\end{lem}

\subsection{Concavity of the function $f$}

The following result will play a key role in this work.
We thank Andrew Elvey-Price for the proof.

\begin{lem}\label{lem:concave}
The function $f$ of Lemma~\ref{lem_asympto} is bounded, concave, and is $C^1$ on $\left( 0, \frac{1}{2} \right)$.
\end{lem}

\begin{proof}
Boundedness follows from continuity on $\left[ 0,\frac{1}{2} \right]$. Moreover, since $\lambda$ is analytic and positive on $\left( 0, \frac{1}{2} \right)$, the function $f$ is $C^2$ on $\left( 0, \frac{1}{2} \right)$.
From~\eqref{eq_f}, we get
\[f''(\theta)=\frac 2 \theta +\frac{\lambda'(\theta)}{\theta\lambda(\theta)}.\]
Writing
\[f''(\theta)=\frac 1 \theta\left(2+\left(\frac{\partial \theta}{\partial h}\right)^{-1}\frac{\partial \log(\lambda)}{\partial h}\right),\]
and then using~\eqref{eqn_lambda_vs_theta} and~\eqref{eq_f}, we get
\[f''(\theta)=-\frac{2(1+6h+128h^3)\sqrt{1-4h}}{3h\left(-(1+8h)\sqrt{1-4h}+(1-2h+16h^2)\log\left(\frac{1+\sqrt{1-4h}}{1-\sqrt{1-4h}}\right)\right)}\]
for $\theta\in(0,1/2)$, that is $h\in(0,1/4)$. It suffices to prove that this is negative. The numerator is clearly positive, so it remains to show
\[-(1+8h)\sqrt{1-4h}+(1-2h+16h^2)\log\left(\frac{1+\sqrt{1-4h}}{1-\sqrt{1-4h}}\right)>0.\] 
We first note that for $h \in \left( 0,\frac{1}{4} \right)$
\[\log\left(\frac{1+\sqrt{1-4h}}{1-\sqrt{1-4h}}\right)=\sum_{j=0}^{\infty}\frac{2}{2j+1}(1-4h)^{j+\frac{1}{2}}>2(1-4h)^{\frac{1}{2}}+\frac{2}{3}(1-4h)^{\frac{3}{2}}.\]
Substituting this inequality above yields
\[-(1+8h)\sqrt{1-4h}+(1-2h+16h^2)\log\left(\frac{1+\sqrt{1-4h}}{1-\sqrt{1-4h}}\right)>\frac{1}{3}(1-4h)^2(5-8h)\sqrt{1-4h}>0,\]
as required.
\end{proof}

We will need to use the concavity of $f$ under the following form, which will be useful to estimate the remainder terms in exponentials resulting from applications of Proposition~\ref{prop_asympto_ratio}.
\begin{corr}
\label{corr_diff_f}
Let $(g_n)$ be such that $\frac{g_n}{n}\rightarrow \theta$ with $0<\theta <\frac{1}{2}$. There exists a constant $b_{\theta}$ depending only on $\theta$ such that for $n$ large enough, for all $1\leq s\leq \ell\leq \theta n/2$ and for all $n_1,n_2,h_1,h_2$ satisfying $n_1+n_2=n+\ell$ and $h_1+h_2=g_n-s+1$, we have
\begin{align*}
	n_1 f\left(\tfrac{h_1}{n_1}\right) + n_2 f\left(\tfrac{h_2}{n_2}\right) -nf \left( \frac{g_n}{n} \right)\leq b_{\theta}\ell.
\end{align*}
\end{corr}

\begin{proof}
By concavity of $f$, we have
\[  n_1 f\left(\tfrac{h_1}{n_1}\right) + n_2 f\left(\tfrac{h_2}{n_2}\right) \leq (n+\ell) f \left( \frac{g_n-s+1}{n+\ell} \right). \]
On the other hand, for $n$ large enough, the assumptions imply $\left[\frac{g_n-s+1}{n+\ell},\frac{g_n}{n}\right]\subset \interv$ where $\frac{\theta}{2}+\frac{1}{4}<\frac{1}{2}$, so
\begin{align*}
 n_1 f\left(\tfrac{h_1}{n_1}\right) + n_2 f\left(\tfrac{h_2}{n_2}\right) &\leq (n+\ell) f\left( \frac{g_n}{n} \right) + (n+\ell)\left( \frac{g_n}{n}-\frac{g_n-s+1}{n+\ell} \right) \max_{\interv} |f'|\\
&\leq (n+\ell) f\left( \frac{g_n}{n} \right) + 2 \ell \max_{\interv} |f'|.
\end{align*}
This proves our claim with $b_{\theta}=\max_{[0,1/2]}f+2 \max_{\interv} |f'|$.
\end{proof}

\section{Isoperimetric inequalities}
\label{sec:iso}

The goal of this section is to prove Theorem~\ref{thm_isoperimetric}. For this, let us fix $\theta \in \left( 0,\frac{1}{2} \right)$ and a sequence $(g_n)$ such that $0 \leq  g_n \leq \frac{n+1}{2}$ for all $n$ and $\frac{g_n}{n} \to \theta$. For any $k_1, k_2$ with $k_1+k_2=2n$ and for any length vector $\vec{\ell}$, we will denote by $\C$ the number of \ksep{s} of lengths $\vec{\ell}$ in $\bT$. We will estimate the expectation of this quantity.

\begin{lem}\label{lem_first_moment_short_multicurves}
	Let $b'_{\theta}=b_{\theta}+2\log(6)$, where $b_{\theta}$ is defined in Corollary~\ref{corr_diff_f}, and take $\delta_{\theta}=\frac{\theta}{4(b'_{\theta}+\log(2))}$. Let also $K_{\theta}=\frac{20}{\theta}$. For $n$ large enough, for any $k_1, k_2$ with $k_1+k_2=2n$ and $K_{\theta}\log n\leq k_1\leq k_2$ and for any length vector $\ell$ with $|\vec{\ell}| \leq \delta_{\theta} k_1$, we have
\[\E \left[ \C \right] \leq \frac{1}{2^{|\vec{\ell}|} n^3}. \]
\end{lem}

\begin{proof}
	If there is a $(k_1, k_2)$-separating multicurve $\eta$ on a triangulation $\bT$ with lengths $\vec{\ell}=(\ell_1, \dots \ell_s)$, then the two connected components of its complement are two triangulations of the $\vec{\ell}$-gon\footnote{The starting edge $\vec{e}_0$ of each cycle of $\eta$ induces two edges of $\bT \backslash \eta$ which serve as the distinguished edges of the corresponding polygons -- to match previous conventions, one of these two edges needs to be reoriented with the polygon to its left.}. Moreover, the connected components of $\bT \backslash \eta$  have respectively $k_1$ and $k_2$ internal faces with $k_1+k_2=2n$, and respective genera $h_1$ and $h_2$ with $h_1+h_2=g_n-s+1$.

Therefore, the number of triangulations of genus $g_n$ with $2n$ triangles with a marked \ksep{} of lengths $\vec{\ell}=(\ell_1,\ell_2,\dots,\ell_s)$ is bounded by
\[6n\sum_{h_1+h_2=g_n-s+1}\left|\T_{\vec{\ell}}(k_1,h_1)\right|\left|\T_{\vec{\ell}}(k_2,h_2)\right|.\]
Note that this is not an equality\footnote{If for example a vertex $v_1$ appears on two boundary cycles of $t_1$ and a vertex $v_2$ appears on two boundary cycles of $t_2$, then gluing $t_1$ and $t_2$ with $v_1$ glued to $v_2$ does not yield a triangulation (the neighbourhood of $v_1 \sim v_2$ could look like two disjoint disks with identified centers).}.
By Lemma~\ref{lem_filling_boundaries}, this is bounded by \[6n\sum_{h_1+h_2=g_n-s+1}(6n_1)^{s-1}\tau(n_1,h_1)(6n_2)^{s-1}\tau(n_2,h_2)\]
with $n_i=\frac{k_i+|\vec{\ell}|}{2}$. Therefore, we can write
\begin{align*}
\E \left[ \C \right] &\leq 6n\sum_{h_1+h_2=g_n-s+1}\frac{(6n_1)^{s-1}\tau(n_1,h_1)(6n_2)^{s-1}\tau(n_2,h_2)}{\tau(n,g_n)}\\
&\leq 6^{2|\vec{\ell}|}n\sum_{h_1+h_2=g_n-s+1}\frac{n_1^{s-1}\tau(n_1,h_1)n_2^{s-1}\tau(n_2,h_2)}{\tau(n,g_n)}
\end{align*}

We recall that we have assumed $n_1 \leq n_2$. The idea will now be to estimate the numerator using Lemma~\ref{lem_asympto} if both pieces are macroscopic, and the more precise Proposition~\ref{prop_asympto_ratio} if $n_1$ is much smaller than $n_2$. We recall from Proposition~\ref{prop_asympto_ratio} the definition of the constant $a_{\theta}<1$ and distinguish two cases:
\begin{enumerate}
\item
If $n_1 \geq \min \left( a_{\theta} n_2, \frac{g_n}{6} \right)$, then both $n_1$ and $n_2$ are of order $n$. Hence, using Lemma~\ref{lem_asympto} and $n=e^{o(n)}$, we have 
\begin{align*}
	\E &\left[\C \right] \leq  \\ &\ \ 6^{2|\vec{\ell}|}\sum_{h_1+h_2=g_n-s+1}\frac{n_1^{2h_1+s-1}n_2^{2h_2+s-1}}{n^{2g_n}}\exp \left( n_1 f \left( \frac{h_1}{n_1} \right) + n_2 f \left( \frac{h_2}{n_2} \right) - n f \left( \frac{g_n}{n} \right)+o(n) \right),
\end{align*}
where the $o(n)$ is uniform in $(n_1, n_2, \vec{\ell}, h_1, h_2)$ (see the remark just after Lemma~\ref{lem_asympto}). In the denominator, we can write $2g_n=2h_1+2h_2+2s-2$. Using Corollary~\ref{corr_diff_f} and the inequality $n_1 \leq n_2 \leq n-(1-\delta_{\theta})n_1$, we have
\begin{align*}
	\E \left[ \C \right] &\leq\left(1- \frac{(1-\delta_{\theta})n_1}n\right)^{2g_n}\exp(b'_{\theta}|\vec{\ell}|+o(n))\\
&\leq \exp(-2n_1 (1-\delta_{\theta})\theta +b'_{\theta}|\vec{\ell}| +o(n_1)).
\end{align*}

\item If $n_1 \leq \min \left( a_{\theta} n_2, \frac{g_n}{6} \right)$, then we have
\[g_n-h_2=h_1+s-1\leq n_1+|\vec{\ell}| \leq \frac{g_n}{6}+\delta_{\theta n} \leq \frac{g_n}{2}\]
for $n$ large enough, so Proposition~\ref{prop_asympto_ratio} applies to the ratio $\frac{\tau(n,g_n)}{\tau(n_2, h_2)}$. Combining this with Lemma~\ref{lem_asympto} on $(n_1,h_1)$, we can bound $\E \left[ \C \right]$ by
\[ 6^{2|\vec{\ell}|} n \sum_{\substack{h_1+h_2=g_n-s+1\\ h_1 \leq k_1}} a_{\theta}^{-(h_1+s-1)} \frac{n_1^{2h_1+s-1}n_2^{2h_2+s-1}}{n^{2g_n}}
		e^{
			n_1 f \left( \frac{h_1}{n_1} \right)+ n_2 f \left( \frac{h_2}{n_2} \right) - n  f \left( \frac{g_n}{n} \right) +o(n_1) 
		},
		\]
where the $o(n_1)$ is uniform in $(\vec{\ell}, h_1, h_2)$ as $n, n_1 \to +\infty$ (we need $n_1 \to +\infty$ because of Lemma~\ref{lem_asympto}). Moreover, we have \[a_{\theta}^{-h_1+s-1} n_1^{2h_1+s-1} \leq \left( \frac{n_1}{a_{\theta}} \right)^{2h_1+s-1} \leq n_2^{2h_1+s-1},\]
so, using also Corollary~\ref{corr_diff_f}, we can write
\[\E \left[ \C \right] \leq n\sum_{\substack{h_1+h_2=g_n-s+1\\ h_1 \leq k_1}} \left(\frac{n_2}{n}\right)^{2g_n}\exp \left( b'_{\theta}|\vec{\ell}|+o(n_1) \right).\]
Finally, just like in the first item, we have $n_2 \leq n-(1-\delta_{\theta}) n_1$, so summing over $0 \leq h_1 \leq k_1$ we have
\[\E \left[ \C \right] \leq n k_1 \exp \left( -2n_1(1-\delta_{\theta}) \theta + b'_{\theta}|\vec{\ell}|+o(n_1) \right),\]
where the $o(n_1)$ is uniform in $(\vec{\ell}, h_1, h_2)$ as $n,n_1 \to +\infty$.
\end{enumerate}

Hence in both cases, using $\delta_{\theta}<1/2$, for $n_1$ large enough (and therefore for $n$ large enough), we have
\begin{align}
\E \left[ \C \right] &\leq \frac{n k_1}{2^{|\vec{\ell}|}}\exp \left( -n_1 \theta +(b'_{\theta}+\log(2))|\vec{\ell}| \right) \label{eqn_bound_number_multicurves}\\
&\leq \frac{n^2}{2^{|\vec{\ell}|}} \exp(-n_1 \theta /2) \nonumber \\
&\leq\frac{1}{2^{|\vec{\ell}|} n^3}, \nonumber
\end{align}
where the second inequality comes from $|\vec{\ell}| \leq 2\delta_{\theta} n_1$ and the choice of $\delta_{\theta}$, and the third from $n_1 \geq \frac{K_{\theta}}{2} \log n$ and the choice of $K_{\theta}$.
\end{proof}

\begin{proof}[Proof of Theorem~\ref{thm_isoperimetric}]
Let $\mathcal C_n$ be the event that there exists a \ksep{} of total length $\ell$ in $\bT$ for some $K_{\theta}\log n\leq k_1\leq k_2$ satisfying $k_1+k_2=2n$ and $\ell \leq \delta_{\theta} k_1$.
We have
\begin{align}\label{eqn_event_Cn}
\P \left( \mathcal C_n \right) &=\P\left(\bigcup_{k_1+k_2=2n\atop K_{\theta}\log n\leq k_1\leq k_2}\bigcup_{|\vec{\ell}| \leq \delta_{\theta} k_1} \left\{ \C\geq 1 \right\} \right)\\
&\leq \sum_{k_1+k_2=2n\atop K_{\theta}\log n\leq k_1\leq k_2}\sum_{|\vec{\ell}| \leq \delta_{\theta} k_1} \frac{1}{2^{|\vec{\ell}|} n^3} \nonumber
\end{align}
for $n$ large enough by Lemma~\ref{lem_first_moment_short_multicurves}. For any $\ell$, the number of length vectors $\vec{\ell}$ with $|\vec{\ell}|=\ell$ is $2^{\ell-1}$, so $\P \left( \mathcal C_n \right) \leq \sum_{k_1+k_2=2n} \sum_{\ell=1}^n \frac{1}{2n^3} \leq \frac{1}{n}$, which concludes the proof.
\end{proof}

\section{Distances}
\label{sec:dist}

\subsection{Lower bounds}

Our goal is now to prove Theorem~\ref{thm_diameter} on the diameter of high genus triangulations. We start with the lower bound.

\begin{prop}\label{prop_lower_bound}
For any $\theta \in \left[ 0,\frac{1}{2} \right)$, there is a constant $c_{\theta}>0$ such that the following holds. Let $(g_n)$ be a sequence such that $0 \leq g_n \leq \frac{n+1}{2}$ for all $n$ and $\frac{g_n}{n} \to \theta$. Conditionally on $\bT$, let $x_n,y_n$ be two uniform independent vertices of $\bT$. Then we have
\[ \P \left( d_{\bT}(x_n,y_n) \geq c_{\theta} \log n \right) \xrightarrow[n \to +\infty]{} 1.\]
\end{prop}

\begin{proof}
The proof consists of a first moment computation on short paths in $\bT$. More precisely, we denote by $\PP$ the number of simple (i.e. vertex-injective) paths of length $\ell$ in $\bT$ between $x_n$ and $y_n$.

Let $(t,\vec{e}_0,x,y, \gamma)$ be a triangulation with $2n$ faces and genus $g$, rooted at $\vec{e}_0$, equipped with two marked vertices $x,y$ and a simple path $\gamma$ from $x$ to $y$. When we slit $\gamma$ open into a boundary of length $2\ell$, we obtain a triangulation of the $2\ell$-gon that we root at the boundary edge started from $x$ which has the $2\ell$-gon on its left.
Considering the original root $\vec{e}_0$ as an additional marked oriented edge, we obtain a map in $\T_{(2\ell)}(2n,g)$ with a marked oriented edge. This operation is injective. Therefore, we have
\begin{align*}
\E \left[ \PP \right] & \leq\frac{6n|\T_{(2\ell)}(2n,g_n)|}{(n+2-2g)^2\tau(n,g_n)} \\
& \leq \frac{6+o(1)}{(1-2\theta)^2}\frac{|\T_{(2\ell)}(2n,g_n)|}{n\tau(n,g_n)}\\
& \leq \frac{6+o(1)}{(1-2\theta)^2}\frac{\tau(n+\ell,g_n)}{n\tau(n,g_n)},
\end{align*}
where the $o(1)$ is uniform in $\ell$, and the last inequality comes from Lemma~\ref{lem_filling_boundaries}. Using Lemma~\ref{lem_ratio} and assuming $\ell=O \left( \log n \right)$, we get
\begin{align*}
\E \left[ \PP \right] \leq \frac{6+o(1)}{(1-2\theta)^2}\frac{1}{n\lambda(\theta)^{\ell+o(\ell)}}.
\end{align*}
Finally, if $\ell\leq \frac{\log n}{2\log(1/\lambda(\theta))}$, this becomes
\begin{align*}
\E \left[ \PP \right] \leq \frac{1}{n^{1/2+o(1)}},
\end{align*}
where the $o(1)$ is uniform in $\ell$. This concludes the proof by a union bound over $\ell \leq \frac{\log n}{2\log(1/\lambda(\theta))}$.
\end{proof}

\subsection{Diameter}

We now move on to the upper bound on the diameter. For this, we will rely mostly on the isoperimetric inequality obtained in Theorem~\ref{thm_isoperimetric}. More precisely, applying it to balls in $\bT$, we will obtain the following intermediate result. We recall that $\delta_{\theta}, K_{\theta}>0$ are given by Theorem~\ref{thm_isoperimetric}.

\begin{lem}\label{lem_ball_growth}
Let $\theta \in \left( 0, \frac{1}{2} \right)$ and let $(g_n)$ be a sequence such that $0 \leq g_n \leq \frac{n+1}{2}$ for all $n$ and $\frac{g_n}{n} \to \theta$. With probability $1-o(1)$ as $n \to +\infty$, for every vertex $x$ of $\bT$ and every $r \geq K_{\theta} \log n$, we have
\begin{enumerate}
\item either $|\partial \B_r(x)|\geq \frac{\delta_{\theta}}{2}|\B_r(x)|$,
\item or $|\B_{r+K_{\theta}\log n}(x)|\geq \frac{4n}{3}$.
\end{enumerate}
\end{lem}

\begin{proof}
We may assume that the conclusion of Theorem~\ref{thm_isoperimetric} holds, i.e. the event $\mathcal{C}_n$ of~\eqref{eqn_event_Cn} does not occur.

If this is the case, let $r\geq K_{\theta} \log n$ and let $x$ be a vertex of $\bT$ such that $|\B_r(x)| < \frac{4n}{3}$ (if not, the second item holds). The ball $\B_r(x)$ splits its complement in $\bT$ into several face-connected components. Let $\mathcal A$ (resp. $\mathcal B$) be the set of those components which have at least (resp. less than) $K_{\theta} \log n$ triangles. Let also $|\mathcal{A}|$ (resp. $|\mathcal{B}|$) be the total number of internal faces of the components of $\mathcal{A}$ (resp. $\mathcal{B}$).
\begin{enumerate}
\item If $|\mathcal A|\geq \frac{1}{2}|\B_r(x)|$, since $|\B_r(x)|\geq r\geq K_{\theta}\log n$, we can apply condition $\mathcal C_n$ to each component $c$ of $\mathcal{A}$:
\[|\partial \B_r(x)|\geq |\partial \mathcal A| \geq \delta_{\theta} \sum_{c \in \mathcal{A}} \min\left(|\B_r(x)|,|c|\right). \]
If one of the terms in the sum is equal to $|\B_r(x)|$, the conclusion is immediate. If not, we get
\[|\partial \B_r(x)|\geq \delta_{\theta} |\mathcal{A}| \geq \frac{\delta_{\theta}}{2} |\B_r(x)|.\]
\item If $|\mathcal A|< \frac{1}{2}|\B_r(x)|$ we note that $\mathcal B \subset \B_{r+K_\theta\log n}(x)$, so
\[ \left| \B_{r+K_\theta\log n}(x) \right| \geq |\B_r(x)|+|\mathcal B|=2n-|\mathcal A|\geq 2n-\frac{1}{2}|\B_r(x)| \geq \frac{4n}{3}.\]
\end{enumerate}
\end{proof}

To finish the proof of Theorem~\ref{thm_diameter}, we just need to apply the last lemma to the growth of two balls around two vertices $x$ and $y$, until both of their volumes exceed $\frac{4n}{3}$, so that the two balls must intersect.

\begin{proof}[Proof of Theorem~\ref{thm_diameter}]
The lower bound is given by Proposition~\ref{prop_lower_bound}. For the upper bound, we write $C'_{\theta}=2K_{\theta}+\frac 1{\log\left(1+\delta_{\theta}/6\right)}$, and $d=\left\lceil\frac{\log (4n/3)}{\log\left(1+\delta_{\theta}/6\right)}\right\rceil$. With probability $1-o(1)$ as $n \to+\infty$, the conclusion of Lemma~\ref{lem_ball_growth} holds, so for every vertex $x$ of $\bT$:
\begin{itemize}
\item either for all $r\in [K_{\theta}\log n, K_{\theta}\log n+d]$ we have $|\partial B_r(x)|\geq \frac{\delta_{\theta}}{2} \left| \B_r(x) \right|$ which, combined with Lemma~\ref{lem:ball_r+1}, implies
\[| \B_{C_{\theta}'\log n}(x)|\geq  |\B_{K_{\theta}\log n+d}(x)|\geq \left(1+\frac{\delta_{\theta}}{6}\right)^d|\B_{K_\theta\log n}(x)|\geq \left(1+\frac{\delta_{\theta}}{6}\right)^d \geq 4n/3;\]
\item or there exists $r_1 \in [K_{\theta}\log n, K_{\theta}\log n+d]$ such that $|\B_{r_1}(x)|\geq 4n/3$, which implies
\[| \B_{C'_{\theta}\log n}(x)|\geq  |\B_{r_1}(x)|\geq 4n/3.\]
\end{itemize} 
Therefore, for any pair $(x,y)$ of vertices of $\bT$, the balls $\B_{C_\theta'\log n}(x)$ and $\B_{C_\theta'\log n}(y)$ contain at least $\frac{4n}{3}$ triangles each, so they intersect, so
\[d_{\bT}(x,y)\leq 2C_{\theta}'\log n,\] 
which concludes the proof.
\end{proof}

\subsection{Typical distances}

\begin{defn}
    Let $\eps>0$. Let $t$ be a finite triangulation of size $n$ and let $f$ be a face of $t$. We say that $f$ is \emph{$\eps$-isolated} if there is a $(k_1,k_2)$-separating multicurve $\gamma$ on $t$ such that:
    \begin{itemize}
        \item the face $f$ belongs to the connected component of $t \backslash \gamma$ which has size $k_1$,
        \item we have $k_2 \geq \sqrt{n}$,
        \item the total length $\ell$ of $\gamma$ satisfies $\ell \leq \eps \min( k_1, k_2)$.
    \end{itemize}
    We also denote by $\mathrm{Iso}_{\eps}(t)$ the number of $\eps$-isolated faces of $t$.
\end{defn}

\begin{prop}\label{prop_isolated_points}
    For any $\theta \in \left(0,\frac{1}{2} \right)$ and $\eta>0$, there is a constant $\eps>0$ such that the following holds. Let $g_n$ be a sequence such that $0 \leq g_n \leq \frac{n+1}{2}$ and $\frac{g_n}{n} \to \theta$. Then for $n$ large enough, we have
    \[ \P \left( \iso \geq \eta n \right) \leq \eta.\]
\end{prop}

\begin{proof}
	First, by Theorem~\ref{thm_isoperimetric}, it is sufficient to treat the case where $k_1 \leq K_{\theta} \log n$, and in particular $k_1<k_2$. Moreover, if $k_1 \leq \eps^{-1}$, then the length of the multicurve satisfies $\ell \geq 1 \geq \eps k_1$.
    Therefore, we denote by $\mathrm{Iso}^*_{\eps}(t)$ the number of $\eps$-isolated faces of a triangulation $t$ with the additional condition that $\eps^{-1} \leq k_1 \leq K_{\theta} \log n$. Like for Theorem~\ref{thm_isoperimetric}, the proof consists mostly of a first moment computation. More precisely, we have
    \begin{align*}
        \E \left[ \isostar \right] & \leq \sum_{k_1=\eps^{-1}}^{K_{\theta} \log n} \sum_{|\vec{\ell}| \leq \eps k_1} k_1 \E \left[ \C \right]\\
        & \leq \sum_{k_1=\eps^{-1}}^{K \log n} \sum_{|\vec{\ell}| \leq \eps k_1} k_1^2 n \exp \left( - \theta k_1 + b'_{\theta} |\vec{\ell}| \right)
    \end{align*}
    for $n$ large enough by~\eqref{eqn_bound_number_multicurves}. The end of the proof is similar to that of Theorem~\ref{thm_isoperimetric}: the number of $\vec{\ell}$ with $|\vec{\ell}|=\ell$ is $2^{\ell-1}$, so
    \begin{align*}
        \E \left[ \isostar \right] & \leq \sum_{k_1=\eps^{-1}}^{K \log n} \sum_{\ell=1}^{\eps k_1} k_1^2 n \exp \left( - \theta k_1 + \left( b'_{\theta}+\log 2 \right) |\vec{\ell}| \right)\\
        & \leq n \sum_{k_1=\eps^{-1}}^{+\infty} k_1^3 \exp \left( -\frac{\theta k_1}{2} \right)
    \end{align*}
    if $\eps$ is chosen small enough to have $(b'_{\theta}+\log 2)\eps < \theta/2$. Finally, the sum $\sum_{k_1=\eps^{-1}}^{+\infty} k_1^3 \exp \left( -\frac{\theta k_1}{2} \right)$ goes to $0$ as $\eps \to 0$, so we can choose $\eps$ such that it is smaller than $\eta^2$, which concludes the proof by the Markov inequality.
\end{proof}

\begin{proof}[Proof of Theorem~\ref{thm_dist}]
    We first notice that we only need to prove the tightness of
    \begin{equation}\label{eqn:tightness_3_points}
    d(x_n, y_n)-d(x_n,u_n)
    \end{equation}
    and then apply it to the triples $(x_n,y_n,u_n)$ and $(u_n,x_n,v_n)$ (up to replacing the constant $M(\eta)$ by $2M(\eta/2)$). For this, the argument will be pretty similar to the proof of Lemma~\ref{lem_ball_growth} and Theorem~\ref{thm_diameter}: we will find $M$ such that if $\left| B_r(x_n) \right| \geq 2\eta n$, then $\left| B_{r+M}(x_n) \right| \geq (1-\eta)2n$.
    
    For this, let $\eta>0$. By Proposition~\ref{prop_isolated_points}, let $\eps>0$ be such that with probability at least $1-\eta$, the triangulation $\bT$ contains at most $\eta n$ isolated faces. If this event occurs, let $r$ be such that $2\eta n \leq \left| B_r(x_n) \right| \leq (1-\eta)2n$. Let $\mathcal{A}$ (resp. $\mathcal{B}$) be the set of connected components $c$ of $\bT \backslash B_r(x_n)$ such that $|\partial c| \geq \eps |c|$ (resp. $|\partial c| < \eps |c|$). Let also $|\mathcal{A}|$ (resp. $|\mathcal{B}|$) be the total number of internal faces of components of $\mathcal{A}$ (resp. $\mathcal{B}$). Then we have
    \[ \left| \partial B_r(x_n) \right| \geq \eps \left| \mathcal{A} \right| = \eps \left( 2n-\left| B_r(x_n) \right| - \left| \mathcal{B} \right| \right) \geq \eps \left( 2\eta n - \eta n \right), \]
    because all the internal faces of $\mathcal{B}$ are $\eps$-isolated. By Lemma~\ref{lem:ball_r+1}, it follows that
    \[ \left| B_{r+1}(x_n) \right| \geq \left| B_r(x_n) \right| + \frac{\eps \eta}{3} n. \]
    In particular, if the random variable $R_n$ denotes the smallest radius $r$ such that $\left| B_r(x_n) \right| \geq 2\eta n$, this implies $|B_{R_n+M} (x_n)| \geq (1-\eta)2n$ for $M=\frac{6}{\eps \eta}$. In this case, we know that $B_{R_n-1}(x_n)$ contains at most $2\eta n$ internal faces, hence at most $6\eta n$ vertices. The same is true for $\bT \backslash B_{R_n+M}(x_n)$. Therefore, for $n$ large enough, using that $y_n$ and $u_n$ are uniform, we have
    \begin{align*}
    \P \left( d(x_n,y_n)-d(x_n,u_n)>M \right) & \leq \P \left( \iso \geq \eta n \right)\\
    & + \P \left( \iso \leq \eta n \mbox{ and } y_n \notin B_{R_n+M}(x_n) \right)\\
    & + \P \left( u_n \in B_{R_n-1}(x_n) \right)\\
    & \leq \eta + \frac{6\eta n}{n+2-2g_n} + \frac{6\eta n}{n+2-2g_n}\\
    & \leq \frac{14 \eta}{1-2\theta},
    \end{align*}
    where $M$ depends only on $\eta$ and $\theta$. This proves the tightness of~\eqref{eqn:tightness_3_points}, and the Theorem.
\end{proof}

\section{Existence of local tentacles and Cheeger constant}
\label{sec:tentacles}

In this section, we will prove the existence of long (i.e. of logarithmic size) path-like objects that we call tentacles. This will entail that the Cheeger constant is $O\left(\frac{1}{\log n}\right)$ and that the diameter differs from the typical distances.

Let us start with a precise definition of the tentacles we will be interested in. Given an edge in a triangulation, a \emph{$k$-insertion} consists in opening this edge into a $2$-gon, and triangulating this $2$-gon by a sequence of $k-1$ edges and $k$ paths of length $2$ as in Figure~\ref{fig:tentacles}, for some $k\geq 1$. We authorize this insertion even if the edge is a loop.
We will call \emph{tentacle} a triangulation that can be obtained recursively by repeated such insertions, starting from a given edge. Note that a $k$-insertion creates $2k$ new inner edges in which an insertion can be performed. Therefore, it is clear inductively that tentacles with $2n$ faces are in bijection with even rooted plane trees with $2n$ edges, where an even rooted plane tree is a rooted plane tree in which each vertex has an even number of children.

\newcommand{\bC}{\mathbf{C}_{2n,g_n}}
A \emph{maximal tentacle} in $\bT$ is a tentacle that is maximal for inclusion. The boundary of this tentacle is a two-gon, which is possibly degenerated (i.e. the two vertices of the two-gon may be the same). We let $T_n$ be the number of maximal tentacles of $\bT$ and $M_n$ be their total number of internal faces (in particular $M_n$ is even).
We also let $\bC$ be the \emph{core} of $\bT$, obtained from $\bT$ by removing all maximal tentacles of $\bT$ and closing each of the remaining $2$-gons into an edge. Note that each maximal tentacle corresponds to one edge of the core. In particular, a maximal tentacle can consists of a single edge of $\bT$ if that edge is not included in a larger tentacle.

\begin{lem}\label{lemma:counttentacles}
	There are constants $t_{\theta}, m_{\theta}>0$ such that with high probability
	$$
	T_n > t_\theta n \quad \mbox{and} \quad M_n > m_\theta n.
	$$
\end{lem}

\begin{proof}
    This will follow from a second moment method on a specific kind of tentacle. More precisely, let $A$ be the triangulation of the $2$-gon with $10$ faces shown in Figure~\ref{fig:tentacles}, and let $N_A$ be the number of copies of $A$ in $\bT$. We highlight that in $N_A$, we also count the occurrences of $A$ where its two boundary vertices are the same. 
    We now assume $n>20$, and note that a triangulation with $2n$ faces with a marked copy of $A$ which does not contain the root edge can be transformed into a triangulation with $2n-10$ faces and a marked edge which is not the root. This transformation is bijective (the inverse operation consists in splitting the marked edge open and inserting $A$, which works even if the marked edge is a loop). Therefore, it follows directly from Lemma~\ref{lem_ratio} that
	$\mathbb{E}[N_A/n] = \frac{3\tau(n-5,g) (1+O(\frac{1}{n}))}{\tau(n,g)}\rightarrow 3 \lambda(\theta)^5,
	$
	where the error term $O(\frac{1}{n})$ accounts for the case where the root edge is in a copy of $A$.
Using the same argument with two marked copies of $A$, we have similarly, 
	$\mathbb{E}[(N_A/n)^2] \rightarrow \frac{9\tau(n-10,g) (1+O(\frac{1}{n}))}{\tau(n,g)}\rightarrow 9 \lambda(\theta)^{10}$.
	It follows directly by the Chebyshev inequality that $N_A> \frac{3}{2}\lambda(\theta)^5n$ w.h.p.. Finally, we note that the copy of $B_4$ (see Figure~\ref{fig:tentacles}) lying inside each copy of $A$ is necessarily a maximal tentacle\footnote{for example because in a tentacle all internal vertices have even degree, whereas the two outer vertices of the $B_4$ inside $A$ have degrees $5$ and $7$} and that those tentacles are disjoint. It follows that $T_n\geq N_A$ and that $M_n\geq 6N_A$,
    and the lemma follows.
\end{proof}

\begin{figure}
	\centering
	\includegraphics[width=\linewidth]{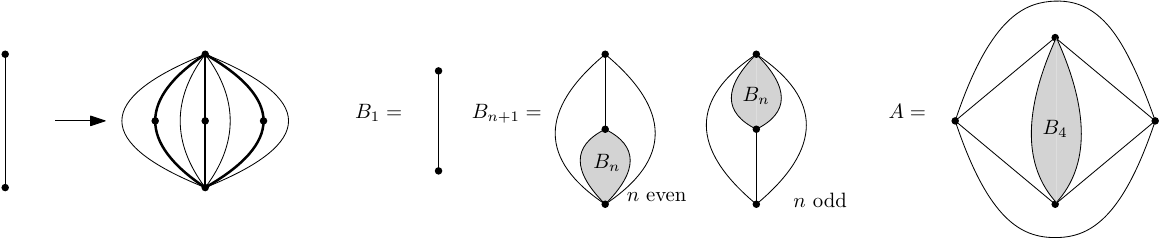}
	\caption{Left: a $k$-insertion, with here $k=3$. The $2k$ fat edges are the ones on which a new insertion can be performed recursively in the creation of tentacles. Center: The "ladder" tentacle $B_n$. Right: The map $A$ used in the proof of Lemma~\ref{lemma:counttentacles}. The copy of $B_4$ inside it is a maximal tentacle.
	}
	\label{fig:tentacles}
\end{figure}

We will now prove that some of the tentacles of $\bT$ are exceptionally long. For each $i\geq 1$ we define $B_i$ as the "ladder" tentacle shown on Figure~\ref{fig:tentacles}. 

\begin{lem}\label{lemma:largeladder}
	There is $s_\theta>0$ such that, w.h.p., one of the maximal tentacles of $\bT$ contains $B_{\ell}$ for $\ell=\lfloor s_\theta \log(n) \rfloor$.
\end{lem}
\begin{proof}
    In all this proof, we will reason conditionally on $\left( T_n, M_n, \bC \right)$. We denote by $(e_1, \dots, e_{T_n})$ the non-root edges of $\bC$ and by $t_i$ the tentacle that fills $e_i$ in $\bT$. Then the family $(t_i)$ is uniform among all families of $T_n$ tentacles with $M_n$ internal faces in total.
    
	We write $T_n=an$, $M_n=2bn$. By Lemma~\ref{lemma:counttentacles} we can assume that $a$ and $b$ are bounded away from zero. Families of $an$ tentacles with a total of $2bn$ faces are in bijection with forests of $an$ even rooted plane trees with a total of $2bn$ edges. By the Lukasiewicz encoding (see e.g.~\cite[L.~5.3.9]{Sta99}), these objects are in bijection with walks with steps $+2$ and $-1$ of total length $(a+3b)n$ ending at position $-an$, and staying above $-an$ before their last point. Note that such walks have $bn$ steps $+2$ and $(a+2b)n$ steps $-1$. We denote by $(V_{i})_{1 \leq i \leq (a+3b)n}$ a random walk picked uniformly among such walks.

	Now let also $(W_{i})_{1 \leq i \leq (a+3b)n}$ be a random walk of length $(a+3b)n$ where each step has increment $+2$ with probability $\frac{b}{a+3b}$ and $-1$ with probability $\frac{a+2b}{a+3b}$, independently. By standard binomial estimates such a walk ends at $-an$ with probability $\Omega(n^{-1/2})$, and by the cycle lemma, conditionnally on this, the probability that it reaches $-an$ for the first time at time exactly $(a+3b)n$ is exactly $\frac{a}{a+3b}>0$ (this is the Kemperman formula). It follows that $V$ has the law of $W$ conditioned on an event of probability $\Omega(n^{-1/2})$.
 
    For $\ell \geq 1$, we now denote by $w_{\ell}$ the Lukasiewicz encoding of the pattern $B_{\ell}$. For $\ell$ odd, this word consists of $\frac{\ell-1}{2}$ times the steps "$+2,+2,-1$" followed by $\frac{3\ell-1}{2}$ steps of $-1$. In particular, it has length $3\ell-2$ with $\ell-1$ steps of $+2$ and $2\ell-1$ steps of $-1$.
	Splitting the $(a+3b)n$ steps of $W$ into intervals of length $3\ell-2$, we get that the probability that $W$ does not contain $w_{\ell}$ is at most	
	\begin{equation}\label{eqn:bound_pattern_lukasiewicz}
    \left( 1-\left( \frac{b}{a+3b} \right)^{\ell-1} \left( \frac{a+2b}{a+3b} \right)^{2\ell-1} \right)^{\lfloor \frac{(a+3b)n}{3\ell-2} \rfloor}\leq \exp \left( -\frac{c'n}{\ell} c^{3\ell} \right)
	\end{equation}
	with $c=\tfrac{a+2b}{a+3b}<1$ and $c'>0$. Taking $\ell$ odd such that $|\ell-s_\theta \log(n)| \leq 1$ with $s_\theta < \frac{1}{3\log(1/c)}$, the quantity in the exponential diverges polynomially in $n$, so~\eqref{eqn:bound_pattern_lukasiewicz} goes to zero much faster than $\Omega(n^{-1/2})$. Therefore, with high probability the walk $V$ contains the pattern $w_{\ell}$. Moreover, the form of the word $w_{\ell}$ implies that its occurrence in $V$ is contained in one excursion of $V$ above its running minimum, so it must correspond to a part of one of the tentacles $t_i$. In particular, one of those tentacles contains $B_{\ell}$.
\end{proof}

\begin{rem}
It may seem at first sight that our notion of tentacle is unnecessarily complicated, and that we could have just defined a tentacle as a copy of $B_{\ell}$ for some $\ell$. The reason why we could not do so is that this does not allow a proper "core and tentacles" decomposition: the triangulation obtained after removing all the copies of $B_{\ell}$ for all $\ell$ may still contain copies of $B_{\ell}$, but then if we condition on the core, we know that nontrivial tentacles must be glued on them. The class of tentacles that we consider is the smallest class which avoids this problem. On the other hand, defining a tentacle as a planar triangulation delimited by a digon would have been natural, but the proof of Lemma~\ref{lemma:largeladder} would then require nontrivial enumerative estimates. 
\end{rem}
We are now able to estimate the Cheeger constant.

\begin{proof}[Proof of Theorem~\ref{thm_cheeger}]
    The upper bound follows immediately from Lemma~\ref{lemma:largeladder}, as the copy of $B_{\ell-2}$ inside of $B_{\ell}$ has $\ell-1$ vertices but only $3$ boundary edges.

    On the other hand, the lower bound will be a straight consequence of Theorem~\ref{thm_isoperimetric}. Let $V=V_1 \cup V_2$ be a partition of the vertices of $\bT$. We want to prove
    \begin{equation}\label{eqn:cheeger_goal}
    |\partial_{\mathrm{Ver}} V_1| \geq \frac{c_{\theta}}{\log n} \min \left( |V_1|, |V_2| \right)
    \end{equation}
    for some $c_{\theta}>0$ which depends only on $\theta$. We denote by $V_1^i$ the vertex-connected components of $V_1$. Moreover, for all $i$, we denote by $F_1^i$ the set of those faces of $\bT$ which have $2$ or $3$ of their vertices in $V_1^i$. Up to exchanging the roles of $V_1$ and $V_2$, we may assume $|F_1^i|\leq n$ for all $i$.

    Now for any edge $e$ such that exactly one side of $e$ is in $F_1^i$, at least one endpoint of $e$ belongs to $V_1^i$ and at least one does not, which means that $|\partial_{\mathrm{Ver}} V_1^i| \geq |\partial F_1^i|$. Moreover, by Theorem~\ref{thm_isoperimetric}, any face-connected component $F_1^{i,j}$ of $F_1^i$ satisfies either $|F_1^{i,j}|\leq K_{\theta} \log n$ or $|\partial F_1^{i,j}| \geq \delta_{\theta}|F_1^{i,j}|$ (because $|F_1^{i,j}| \leq |F_1^i| \leq n$). Summing over $j$, we get
    \[ |\partial_{\mathrm{Ver}} V_1^i| \geq |\partial F_1^i| \geq \min \left( \delta_{\theta},\frac{K_{\theta}^{-1}}{\log n} \right) |F^1_i|. \]
    Moreover, we recall that $V_1^i$ is connected so if $|V_1^i| \geq 2$, the number of edges with both endpoints in $V_1^i$ is at least $|V_1^i|/2$. Each of those is incident to at least one face of $F_1^i$, which implies $|F_1^i| \geq |V_1^i|/6$, so
    \begin{equation}\label{eqn_cheeger_connected_component}
    |\partial_{\mathrm{Ver}} V_1^i| \geq \frac{\min(\delta_{\theta}, K_{\theta}^{-1})}{6\log n} |V^1_i|.
    \end{equation}
    This last inequality is also true if $|V_1^i|=1$, so we can finally sum~\eqref{eqn_cheeger_connected_component} over $i$ to obtain~\eqref{eqn:cheeger_goal}.
\end{proof}

Note that it is also an easy consequence of Theorem~\ref{thm_isoperimetric} and Lemma~\ref{lemma:largeladder} that the Cheeger constant of the \emph{dual} graph of $\bT$ is also of order $\frac{1}{\log n}$ with high probability.

\section{Conjectures on optimal constants}\label{sec:conjectures}

Theorems~\ref{thm_diameter} and~\ref{thm_isoperimetric} give a meaning to the "hyperbolic" nature of $\bT$ at a global scale. However, we have not tried to obtain sharp values for the constants in these results, and we expect that doing so should be a difficult problem. On the other hand, it is possible to conjecture what the optimal constants in Theorems~\ref{thm_diameter} and~\ref{thm_dist} should be. We recall from~\cite{BL19} that the parameter $h_{\theta} \in \left( 0,\frac{1}{4} \right)$ is linked to $\theta$ by the formula
\[ \frac{1-2\theta}{6} = \frac{h_{\theta} \log \frac{1+\sqrt{1-4h_{\theta}}}{1-\sqrt{1-4h_{\theta}}}}{(1+8h_{\theta})\sqrt{1-4h_{\theta}}}. \]
Moreover, following~\cite{B18}, we write
\[ m_{\theta}=\frac{1-2h_{\theta}-\sqrt{1-4h_{\theta}}}{2h_{\theta}} \in (0,1). \]
In particular, the rate of exponential volume growth of the ball of radius $r$ in the local limit of $\bT$ is $m_{\theta}^{-r}$. We can now formulate precise conjectures.
\begin{conj}\label{conj:optimal_constants}
Let $x_n, y_n$ be two uniform, independent vertices of $\bT$. Then we have the convergences in probability
\begin{equation}\label{eqn:convergence_d1_d2}
\frac{1}{\log n} d_{\bT}(x_n,y_n) \xrightarrow[n \to +\infty]{P} D_{\theta} \quad \mbox{and} \quad \frac{1}{\log n} \text{diam}(\bT) \xrightarrow[n \to +\infty]{P} D'_{\theta},
\end{equation}
where
\[ D_{\theta}=\frac{1}{\log(m_{\theta}^{-1})} \quad \mbox{and} \quad D'_{\theta}=\frac{3}{\log (m_{\theta}^{-1})}. \]
Moreover, if we denote by $R_{\mathrm{plan}}(x_n)$ the planarity radius around $x_n$, i.e. the largest $r$ such that the ball $B_r(x_n)$ is planar (see~\cite{Louf}), then
\[ \frac{1}{\log n} R_{\mathrm{plan}}(x_n) \xrightarrow[n \to +\infty]{P} \frac{1}{2}D_{\theta}. \]
\end{conj}

The conjecture for $D_{\theta}$ comes from extrapolating to larger scales the rate of exponential growth of balls at the local scale in $\bT$. Moreover, we believe that $D'_{\theta}-D_{\theta}$ is given by twice the length of the longest tentacle of $\bT$. We also expect that the tentacles should behave roughly like i.i.d. Boltzmann planar triangulations with Boltzmann weight $\lambda(\theta)$ given by~\cite{BL19}. A description of distances to the root in such triangulations is given in~\cite{B18} using the Krikun decomposition, which is where the conjecture for $D'_{\theta}$ comes from. Finally, our conjecture for the planarity radius comes from the analogy with random graphs, where non-planarity appears when more than $\sqrt{n}$ vertices have been explored.

Finally, it may sound surprising to expect $D'_{\theta}>D_{\theta}$, as it is not the case for $3$-regular graphs~\cite{BFdlV}. However, this is a common behaviour as soon as tentacles are not prohibited by the local structure of the graph. For example, the giant component of a supercritical Erd\"os-R\'enyi random graph exhibits the same behaviour~\cite{RW10}, for the same reason as here. This was also already conjectured in~\cite{Ray13a} for unicellular maps. By the same argument as the lower bound of Theorem~\ref{thm_cheeger} (existence of logarithmic tentacles), we are able to prove that if~\eqref{eqn:convergence_d1_d2} is true, then $D'_{\theta}>D_{\theta}$. 

\begin{prop}\label{prop:d1_not_d2}
	There exists $t_\theta>0$ such that if $x_n,y_n$ are the starting points of two independent uniform oriented edges in $\bT$, then
	$$
	\text{diam}(\bT)-d_{\bT}(x_n,y_n) \geq t_\theta \log n
	$$
	with high probability.
\end{prop}

\begin{proof}
    The idea of this proof is that the base of a large tentacle becomes a typical edge if the tentacle is removed. We first notice that Theorem~\ref{thm_dist} remains true if the uniformly chosen vertices $x_n,y_n,u_n,v_n$ are replaced by the starting points of uniformly chosen oriented edges $\vec{e}_1, \dots, \vec{e}_4$. The proof is exactly the same, the only difference is that in the end, instead of counting vertices, we need to notice that $B_{R_n-1}(x_n)$ and $\bT \backslash B_{R_n+M}(x_n)$ contain at most $6\eta n$ edges each. We will write $d_{\bT}(\vec{e}_1,\vec{e}_2)$ for the distance in $\bT$ between the starting points of $\vec{e}_1$ and $\vec{e}_2$.

    For any two oriented edges $\vec{e}_1, \vec{e}_2$ of $\bT$, we have
    \begin{equation}\label{eqn:distance_decomposition}
    d_{\bT} \left( \vec{e}_1, \vec{e}_2  \right) = H(\vec{e_1})+H(\vec{e}_2)+d_{\bC} \left( p(\vec{e}_1), p(\vec{e}_2) \right)+\eps_{\vec{e}_1,\vec{e}_2},
    \end{equation}
    where $p(\vec{e})$ is the base of the tentacle that $\vec{e}$ belongs to, $H(\vec{e})$ is the graph distance between $\vec{e}$ and $p(\vec{e})$, and $\eps_{\vec{e}_1,\vec{e}_2} \in [-2,2]$. Now let $\vec{e}_*$ be an oriented edge of $\bT$ which maximizes $H(\vec{e})$ (if there are several maximizers, pick one uniformly at random), so that $H(\vec{e}_*) \geq \frac{s_{\theta}}{2} \log n$ w.h.p. by Lemma~\ref{lemma:largeladder}. Let also $\vec{e}_0, \vec{e}_1, \vec{e}_2$ be three independent uniformly distributed oriented edges of $\bT$. With high probability, we have
    \[|d_{\bT}(\vec{e}_0,\vec{e}_1)-d_{\bT}(\vec{e}_0,\vec{e}_2)| \leq \log \log n\]
    by Theorem~\ref{thm_dist}, and $H(\vec{e}_i) \leq \log \log n$ for $i \in \{0,1,2\}$ by Proposition~\ref{prop_isolated_points} (if not $\vec{e}_i$ would be $\frac{2}{\log \log n}$-isolated). Using~\eqref{eqn:distance_decomposition} for $(\vec{e}_0,\vec{e}_1)$ and for $(\vec{e}_0,\vec{e}_2)$, it follows that
    \begin{equation}\label{eqn:concentration_distances_core}
    \left| d_{\bC} \left( p(\vec{e}_0), p(\vec{e}_1) \right)-d_{\bC} \left( p(\vec{e}_0), p(\vec{e}_2) \right) \right| \leq 2 \log \log n+2.
    \end{equation}
    On the other hand, conditionally on $\bC$ and the multiset of tentacles, the tentacles are glued uniformly at random on the edges of $\bC$. It follows that up to an event of probability $o(1)$ (namely that two of $\vec{e}_0, \vec{e}_1, \vec{e}_2, \vec{e}_*$ belong to the same tentacle, which is unlikely by Theorem~\ref{thm_isoperimetric}), the edges $p(\vec{e}_0), p(\vec{e}_1), p(\vec{e}_2), p(\vec{e}_*)$ are independent uniform edges of $\bC$. In particular $\left( p(\vec{e}_0), p(\vec{e}_1), p(\vec{e}_2) \right)$ and $\left( p(\vec{e}_0), p(\vec{e}_1), p(\vec{e}_*) \right)$ have the same law conditionally on $\bC$, so we can replace $\vec{e}_2$ by $\vec{e}_*$ in~\eqref{eqn:concentration_distances_core}. Combined with~\eqref{eqn:distance_decomposition} for $(\vec{e}_0, \vec{e}_1)$ and for $(\vec{e}_0, \vec{e}_*)$, this proves
    \[ \diam(\bT) \geq d_{\bT}(\vec{e}_0, \vec{e}_*) \geq d_{\bT}(\vec{e}_0, \vec{e}_1) + \frac{s_{\theta}}{2} \log n -3 \log \log n - 8.\]
\end{proof}

\subsection*{Acknowledgements.} The second author acknowledges funding from the grants ANR-19-CE48-0011 “COMBIN\'E” and ANR-18-CE40-0033 “Dimers”. 
Some ideas of this paper were developped during ongoing work of the third author with Andrew Elvey-Price, Wenjie Fang and Michael Wallner. In particular, the authors thank Andrew Elvey-Price for the proof of Lemma~\ref{lem:concave}.

\bibliographystyle{abbrv}
\bibliography{bibli}

\begin{thebibliography}{10}

\bibitem{ACCR13}
O.~Angel, G.~Chapuy, N.~Curien, and G.~Ray.
\newblock The local limit of unicellular maps in high genus.
\newblock {\em Electron. Commun. Probab.}, 18(86):1--8, 2013.

\bibitem{AS03}
O.~Angel and O.~Schramm.
\newblock Uniform infinite planar triangulations.
\newblock {\em Comm. Math. Phys.}, 241(2-3):191--213, 2003.

\bibitem{B13book}
I.~Benjamini.
\newblock {\em Coarse Geometry and Randomness: {\'E}cole d’{\'E}t{\'e} de
  Probabilit{\'e}s de Saint-Flour XLI -- 2011}.
\newblock Lecture Notes in Mathematics. Springer International Publishing,
  2013.

\bibitem{BM22}
J.~Bettinelli and G.~Miermont.
\newblock Compact {B}rownian surfaces {II}. {O}rientable surfaces.
\newblock {\em arXiv:2212.12511}, 2022.

\bibitem{BFdlV}
B.~Bollob\'{a}s and W.~Fernandez de~la Vega.
\newblock The diameter of random regular graphs.
\newblock {\em Combinatorica}, 2(2):125--134, 1982.

\bibitem{Bud15}
T.~Budd.
\newblock The peeling process of infinite {B}oltzmann planar maps.
\newblock {\em Electronic Journal of Combinatorics}, 23, 06 2015.

\bibitem{B18}
T.~Budzinski.
\newblock Infinite geodesics in hyperbolic random triangulations.
\newblock {\em Ann. Inst. Henri Poincar{\'e}, Probab. Stat.}, 56(2):1129--1161,
  2020.

\bibitem{BCP19}
T.~Budzinski, N.~Curien, and B.~Petri.
\newblock Universality for random surfaces in unconstrained genus.
\newblock {\em Electron. J. Combin.}, 26(4):Paper No. 4.2, 34, 2019.

\bibitem{BL19}
T.~Budzinski and B.~Louf.
\newblock Local limits of uniform triangulations in high genus.
\newblock {\em Invent. Math.}, 223(1):1--47, 2021.

\bibitem{BL20}
T.~Budzinski and B.~Louf.
\newblock Local limits of bipartite maps with prescribed face degrees in high
  genus.
\newblock {\em Ann. Probab.}, 50(3):1059--1126, 2022.

\bibitem{Chapuy:profile}
G.~Chapuy.
\newblock The structure of unicellular maps, and a connection between maps of
  positive genus and planar labelled trees.
\newblock {\em Probab. Theory Related Fields}, 147(3-4):415--447, 2010.

\bibitem{Chapuy:voronoi}
G.~Chapuy.
\newblock On tessellations of random maps and the {$t_g$}-recurrence.
\newblock {\em Probab. Theory Related Fields}, 174(1-2):477--500, 2019.

\bibitem{CD06}
P.~Chassaing and B.~Durhuus.
\newblock Local limit of labeled trees and expected volume growth in a random
  quadrangulation.
\newblock {\em Ann. Probab.}, 34(3):879--917, 2006.

\bibitem{CS04}
P.~Chassaing and G.~Schaeffer.
\newblock Random planar lattices and integrated super{B}rownian excursion.
\newblock {\em Probab. Theory Related Fields}, 128(2):161--212, 2004.

\bibitem{CP16}
S.~Chmutov and B.~Pittel.
\newblock On a surface formed by randomly gluing together polygonal discs.
\newblock {\em Advances in Applied Mathematics}, 73:23--42, 2016.

\bibitem{Ch97}
F.~R.~K. Chung.
\newblock {\em Spectral Graph Theory}.
\newblock American Mathematical Society, 1997.

\bibitem{CurPSHIT}
N.~Curien.
\newblock Planar stochastic hyperbolic triangulations.
\newblock {\em Probability Theory and Related Fields}, 165(3):509--540, 2016.

\bibitem{DKRV16}
F.~David, A.~Kupiainen, R.~Rhodes, and V.~Vargas.
\newblock Liouville quantum gravity on the {R}iemann sphere.
\newblock {\em Communications in Mathematical Physics}, 342(3):869--907, Mar
  2016.

\bibitem{Eynard:book}
B.~Eynard.
\newblock {\em Counting surfaces}, volume~70 of {\em Progress in Mathematical
  Physics}.
\newblock Birkh\"{a}user/Springer, [Cham], 2016.
\newblock CRM Aisenstadt chair lectures.

\bibitem{Gam06}
A.~Gamburd.
\newblock Poisson--{D}irichlet distribution for random {B}elyi surfaces.
\newblock {\em Ann. Probab.}, 34(5):1827--1848, 2006.

\bibitem{GJ08}
I.~P. Goulden and D.~M. Jackson.
\newblock The {KP} hierarchy, branched covers, and triangulations.
\newblock {\em Adv. Math.}, 219(3):932--951, 2008.

\bibitem{JansonLouf1}
S.~Janson and B.~Louf.
\newblock Short cycles in high genus unicellular maps.
\newblock {\em Ann. Inst. Henri Poincar\'{e} Probab. Stat.}, 58(3):1547--1564,
  2022.

\bibitem{JansonLouf2}
S.~Janson and B.~Louf.
\newblock Unicellular maps vs. hyperbolic surfaces in large genus: simple
  closed curves.
\newblock {\em Ann. Probab.}, 51(3):899--929, 2023.

\bibitem{Kri05}
M.~Krikun.
\newblock Local structure of random quadrangulations.
\newblock {\em arXiv:0512304}, 2005.

\bibitem{LG11}
J.-F. Le~Gall.
\newblock Uniqueness and universality of the {B}rownian map.
\newblock {\em Ann. Probab.}, 41:2880--2960, 2013.

\bibitem{Louf}
B.~Louf.
\newblock Planarity and non-separating cycles in uniform high genus
  quadrangulations.
\newblock {\em Probab. Theory Related Fields}, 182(3-4):1183--1206, 2022.

\bibitem{Mar16}
C.~Marzouk.
\newblock Scaling limits of random bipartite planar maps with a prescribed
  degree sequence.
\newblock {\em Random Struct. Algorithms}, 53(3):448--503, 2018.

\bibitem{Mie11}
G.~Miermont.
\newblock The {B}rownian map is the scaling limit of uniform random plane
  quadrangulations.
\newblock {\em Acta Math.}, 210(2):319--401, 2013.

\bibitem{MS15b}
J.~Miller and S.~Sheffield.
\newblock Liouville quantum gravity and the {Brownian} map. {I}: {The}
  {{\(\text{QLE}(8/3,0)\)}} metric.
\newblock {\em Invent. Math.}, 219(1):75--152, 2020.

\bibitem{Mir13}
M.~Mirzakhani.
\newblock Growth of {W}eil-{P}etersson volumes and random hyperbolic surfaces
  of large genus.
\newblock {\em J. Differential Geom.}, 94(2):267--300, 2013.

\bibitem{MP19}
M.~Mirzakhani and B.~Petri.
\newblock Lengths of closed geodesics on random surfaces of large genus.
\newblock {\em Comment. Math. Helv.}, 94(4):869--889, 2019.

\bibitem{Ray13a}
G.~Ray.
\newblock Large unicellular maps in high genus.
\newblock {\em Ann. Inst. H. Poincaré Probab. Statist.}, 51(4):1432--1456, 11
  2015.

\bibitem{RW10}
O.~Riordan and N.~Wormald.
\newblock The diameter of sparse random graphs.
\newblock {\em Combinatorics, Probability and Computing}, 19(5-6):835–926,
  2010.

\bibitem{ShenWu}
Y.~Shen and Y.~Wu.
\newblock Arbitrarily small spectral gaps for random hyperbolic surfaces with
  many cusps.
\newblock {\em arXiv:2203.15681}, 2022.

\bibitem{Sta99}
R.~P. Stanley.
\newblock {\em Enumerative combinatorics. {V}ol. 2}, volume~62 of {\em
  Cambridge Studies in Advanced Mathematics}.
\newblock Cambridge University Press, Cambridge, 1999.
\newblock With a foreword by Gian-Carlo Rota and appendix 1 by Sergey Fomin.

\end{thebibliography}

\end{document}